\documentclass[11pt,letterpaper]{amsart}

\usepackage{amsmath,amssymb,amsthm,amscd}
\usepackage[shortlabels]{enumitem}
\usepackage{mathtools}
\usepackage{color}
\usepackage{bbm}
\usepackage{tikz-cd}

\usepackage[hmargin=2.5cm,vmargin=2cm]{geometry}
\usepackage[normalem]{ulem}
\usepackage{tcolorbox}
\usepackage{mathabx}

\usepackage{hyperref}
\usepackage{float}
\hypersetup{colorlinks=true,linkcolor={black},
citecolor={red}
}

\usepackage{comment}

\newtheorem{theorem}{Theorem}[section]
\newtheorem*{theorem*}{Theorem}
\newtheorem{lemma}[theorem]{Lemma}
\newtheorem*{lemma*}{Lemma}
\newtheorem{corollary}[theorem]{Corollary}

\newtheorem{proposition}[theorem]{Proposition}

\newtheorem{classconj}[theorem]{Classification Conjecture}

\theoremstyle{definition}
\newtheorem{definition}[theorem]{Definition}
\newtheorem{remark}[theorem]{Remark}

\numberwithin{equation}{section}

\newcommand{\K}{\mathsf{k}}

\newcommand{\cF}{\mathcal F}

\def\Nz{\mathbb{N}}

\def\Qz{\mathbb{Q}}

\def\Zz{\mathbb{Z}}

\def\1z{\mathbb{1}}

\newcommand{\Ad}{\textup{Ad}}

\newcommand{\id}{\textup{id}}

\newcommand{\End}{\textup{End}}

\newcommand{\gr}{\textup{gr}}
\newcommand{\Id}{\textup{Id}}
\DeclareMathOperator{\mult}{mult}

\newcommand\Gr[1][]{{\operatorname{{Gr}-}}}
\newcommand{\Modd}{\operatorname{Mod-}}



\begin{document}

\title[The Graded Classification Conjecture]{Unital aligned shift equivalence and the graded classification conjecture for Leavitt path algebras}

\author[K. A. Brix]{Kevin Aguyar Brix}
\address[K. A. Brix]{Department of Mathematics and Computer Science, University of Southern Denmark, 5230 Odense, Denmark}
\email{kabrix@imada.sdu.dk}

\author[A. Dor-On]{Adam Dor-On}
\address[A. Dor-On]{Department of Mathematics, University of Haifa, Mount Carmel, Haifa 3103301, Israel}
\email{adoron.math@gmail.com}

\author[R. Hazrat]{Roozbeh Hazrat}
\address[R. Hazrat]{Centre for Research in Mathematics and Data Sceince,  Western Sydney University, Australia}
\email{r.hazrat@westernsydney.edu.au}

\author[E. Ruiz]{Efren Ruiz}
\address[E. Ruiz]{Department of Mathematics, University of Hawaii, Hilo, 200 W. Kawili St., Hilo, Hawaii,
96720-4091 U.S.A.}
\email{ruize@hawaii.edu}

\thanks{This work was initiated while all four authors were attending the workshop \emph{Combinatorial *-algebras} at the Mathematisches Forschungsinstitut Oberwolfach. 
Brix was supported by a Starting Grant from the Swedish Research Council (2023–-03315) and a Reintegration Fellowship from the Carlsberg Foundation (CF23--1328). 
Dor-On was partially supported by an NSF-BSF grant no. 2530543 / 2023695 (respectively), Horizon Marie-Curie SE project no. 101086394 and a DFG Middle-Eastern collaboration project no. 529300231. 
Hazrat acknowledges Australian Research Council grant DP230103184. 
Ruiz was partially supported by the Simons Foundation and by the Mathematisches Forschungsinstitut Oberwolfach.}
\date{\today}

\begin{abstract}
We prove that a unital shift equivalence induces a graded isomorphism of Leavitt path algebras when the shift equivalence satisfies an alignment condition.
This yields another step towards confirming the Graded Classification Conjecture. 
Our proof uses the bridging bimodule developed by Abrams, the fourth-named author and Tomforde, as well as a general lifting result for graded rings that we establish here.
This general result also allows us to provide simplified proofs of two important recent results: one independently proven by Arnone and Va{\v s} through other means that the graded $K$-theory functor is full, and the other proven by Arnone and Corti\~nas that there is no unital graded homomorphism between a Leavitt algebra and the path algebra of a Cuntz splice.
\end{abstract}

\maketitle

\section{Introduction}
Subshifts of finite type are at the heart of dynamics, appearing naturally as a way of modeling a myriad of dynamical phenomena in e.g. topological quantum field theory, ergodic theory, statistical mechanics and coding and information theory \cite{Lind-Marcus2021}.
The most fundamental problem in symbolic dynamics is the classification of subshifts of finite type up to topological conjugacy, dating back to the seminal work of Williams \cite{Williams1973}. This conjugacy problem may be recast into the algebraic problem of strong shift equivalence for adjacency matrices:
two square matrices $A$ and $B$ are elementary strong shift equivalent if there are non-negative rectangular matrices $R$ and $S$ such that $A = RS$ and $SR = B$;
they are strong shift equivalent if there are $n\in \Nz$ and square matrices $A = A_0,\ldots,A_n = B$ such that $A_i$ and $A_{i+1}$ are elementary strong shift equivalent for every $i < n$.
To this day, it is still not known if strong shift equivalence is decidable.

A much more tractable and closely related notion is that of shift equivalence, also introduced by Williams:
two square matrices $A$ and $B$ are shift equivalent if there are $m\in \Nz$ and non-negative rectangular matrices such that
\[
A^m = RS, \quad AR = RB, \quad BS = SA, \quad B^m = SR.
\]
It was orginally believed that the two notions would coincide, but after 20 years, Kim and Roush found various counterexamples \cite{Kim-Roush1992,Kim-Roush1999,Kim-RoushW1999}.
Understanding the fine difference between the two relations remains an active area of research with the ultimate goal of shedding light on the conjugacy problem above.

Leavitt path algebras were introduced in the 2000's~\cite{Lpabook}, and for every field $\K$ they assigns a $\K$-algebra $L_\K(E)$ to a directed graph $E$.  
Since directed graphs model subshifts of finite type, it was observed early on that strong shift equivalence between matrices induces an equivalence between the categories of graded modules over the corresponding Leavitt path algebras~\cite{HazaratDyn2013,Ara-Pardo}.
Today, one of the main conjectures in this area purports that mere shift equivalence of matrices translates into graded equivalence of the categories of graded modules of Leavitt path algebras.
This provides an important link between the conjugacy problem in dynamics and the following central classification conjecture for non-commutative algebras~\cite{Hazrat2013,HazaratDyn2013,willie}.

\begin{classconj}\label{gcc2013}
Let $E$ and $F$ be finite graphs with essential adjacency matrices $A$ and $B$, respectively. 
The following are equivalent.

\begin{enumerate}[\upshape(1)]

\item \label{cc:1} The matrices $A$ and $B$ are shift equivalent;

\item \label{cc:2} There is an order-preserving $\mathbb Z[x,x^{-1}]$-module isomorphism  $\theta\colon K_0^{\gr}(L_\K(E)) \to K_0^{\gr}(L_\K(F))$; 

\item \label{cc:3} The graded module categories $\Gr L_\K(E)$ and $\Gr L_\K(F)$ are graded Morita equivalent. 
\end{enumerate}

Furthermore, if in condition \ref{cc:2}, the map $\theta$ is pointed, i.e.,  $\theta([L_\K(E)])=[L_\K(F)]$, this is equivalent to the replacement of condition \ref{cc:3} with $L_\K(E)$ and $L_\K(F)$ being graded isomorphic. 
\end{classconj}

The equivalence of \ref{cc:1} and \ref{cc:2} above was established in \cite{HazaratDyn2013, Ara-Pardo}. 
Ara and Pardo also show that if the pointed isomorphism $\theta\colon K_0^{\gr}(L_\K(E)) \to K_0^{\gr}(L_\K(F))$ arises from a strong shift equivalence, then the Leavitt path algebras are graded isomorphic. 

In this note, we establish the unital/pointed version of the implication \ref{cc:1} to \ref{cc:3} when the shift equivalence satisfies a particular alignment condition.
Our methods use advances from \cite{Abrams-Ruiz-Tomforde2023} which followed recent work in \cite{CDE24} in the C*-setting, as well as general lifting results for graded rings that we establish in Section \ref{sec:pointed-graded-ME}.
Our categorical approach also allows us to give another proof of a conjecture made by the third-named author in \cite{Hazrat2013}, predicting that the pointed graded Grothendieck group $K_0^{\gr}$ is full, i.e., any order-preserving, pointed, $\Zz[x,x^{-1}]$-module homomorphism $\alpha \colon K_0^{\mathrm{gr}} (L_\K(E)) \to K_0^{\mathrm{gr}} (L_\K(F))$ is induced by a unital graded homomorphism $\psi \colon L_\K(E) \to L_\K(F)$. 
This provides a new and conceptual approach that significantly simplifies the proofs compared to Arnone \cite{Arnone2023} and Va{\v s} \cite{Vas2023}.
Finally, we also provide a short proof of a result recently established by Arnone and Corti\~nas \cite{ArnoneCor23} that there are no unital graded homomorphisms between a Leavitt algebra $L_n$ and the path algebra of a Cuntz splice of a single vertex with $n$ loops.
It is likely our methods will be relevant for future investigations related to other pertinent problems in the area.

In Section \ref{sec:basics} we address preliminaries and notation, and in Section \ref{sec:pointed-graded-ME} we introduce pointed graded Morita equivalence and establish a general lifting result for graded rings.
We then apply this general result in Section \ref{sec:Ktheory} to give a new and conceptually clean proof of the fullness conjecture, as well as the non existence of unital graded homomorphism between a Leavitt algebra and the Leavitt path algebra of a Cuntz splice. Finally, in Section \ref{sec:aligned} we show that a unital shift equivalence satisfying an alignment condition induces a graded isomorphism of Leavitt path algebras.

\section{Basic definitions} \label{sec:basics}

\subsection{Symbolic dynamics} \label{sec:symbolic-dynamics}
We recommend \cite{Lind-Marcus2021} as an introduction to symbolic dynamics.

Let $A$ be a square matrix over $\Nz$ and let $|A|$ be the size of $A$.
The \emph{dimension group} of $A$, or of the subshift it determines, is the inductive limit of the stationary inductive system given by $A^t$ as a linear map on $\Zz^{|A|}$, i.e., 
\begin{equation}\label{thuluncht}
\mathbb Z^{|A|} \stackrel{A^t}{\longrightarrow} \mathbb Z^{|A|} \stackrel{A^t}{\longrightarrow}  \mathbb Z^{|A|} \stackrel{A^t}{\longrightarrow} \cdots.
\end{equation}

As a group, this is isomorphic to 
\begin{equation}\label{dimgroup1}
    G_A \coloneqq (\Zz^{|A|}\times \Nz)/ \sim_{A^t}
\end{equation}
 where $(v,k) \sim_{A^t} (w,l)$  if $(A^t)^{m-k} v = (A^t)^{m-l} w$, for some $m\in \Nz$. Let $[v,k]$ denote the equivalence class of $(v,k)$. Clearly $[(A^t)^n v,n+k]=[v,k]$. Then, it is not difficult to show that the direct limit of the system \eqref{thuluncht} is the abelian group consists of equivalence classes $[v,k]$, $v\in \mathbb Z^{|A|}$, $k \in \mathbb N$, with addition defined by
\[[v,k]+[w,k']=[(A^t)^{k'}v+(A^t)^kw,k+k'].\]
  
The positive cone in $G_A$ is $G_A^+ \coloneqq \big \{ [v,k] \in G_A : v\in \Nz^{|A|} \big \}$, and multiplication by $A^t$ defines an automorphism $\theta_A$ on $G_A$ that preserves the positive cone.
The class of the unit $u_A \coloneqq [(1,\ldots,1), 0]$ is an order unit in $G_A$.
Alternatively (see \cite[Section 7.4]{Lind-Marcus2021}), the dimension group is defined in terms of the eventual image 
\begin{equation}\label{dimgroup2}
\Delta_A \coloneqq \Big \{ v\in \bigcap_{k=1}^\infty \Qz^{|A|} A^k : v A^l\in \Zz^{|A|} \text{, for some } l\in \Nz \Big \}
\end{equation}
with positive cone $\Delta_A^+ \coloneqq \big \{ v\in \Delta_A : v A^l\in \Nz^{|A|} \text{, for some } l\in \Nz \big \}$,
and multiplication $v\mapsto vA$ defines an automorphism $\delta_A$ on $\Delta_A$ that preserves the positive cone.  The automorphism $\delta_A$ naturally makes $\Delta_A$ into a $\Zz[x,x^{-1}]$-module.  The triple 
$(\Delta_A , \Delta_A^+ , \delta_A)$
is called \emph{Krieger's dimension triple of $A$}.

The two pictures of the dimension data (\ref{dimgroup1}) and (\ref{dimgroup2}) are equivalent. Indeed, a concrete order $\Zz[x,x^{-1}]$-module isomorphism is given by 
\begin{align*}
\psi_A\colon \Delta_A &\longrightarrow G_A\\
v &\longmapsto [(A^t)^l v^t, l], 
\end{align*}
where $l\in \Nz$ is such that $vA^l\in \Zz^{|A|}$.
Note that $\psi_A$ is a unital isomorphism where the order unit in $\Delta_A$ is given by $u \in \bigcap_{k=1}^\infty \Qz^{|A|} A^k$ such that $(1,\ldots, 1) A^{|A|} = u A^{|A|}$.  The existence of $u$ follows from \cite[Remark~7.4.4]{Lind-Marcus2021}. 

Two square matrices $A$ and $B$ over $\Nz$ are \emph{shift equivalent} if there exist $m\in \Nz$ and rectangular matrices $R$ and $S$ over $\Nz$ satisfying the relations
\begin{equation} \label{eq:SE}
    A^m = R S, \quad A R = R B, \quad B^m = S R, \quad B S = S A.
\end{equation}

A shift equivalence implemented via $(R,S)$ then determines the isomorphisms, 
\begin{align}\label{thuemas}
R\colon \Delta_A &\longrightarrow \Delta_B & R^t\colon G_A&\longrightarrow G_B\\
 v&\longmapsto vR &    [v, k] &\longmapsto [R^t v, k]. \notag
\end{align}

A straightforward computation shows that both isomorphisms preserve the positive cones and the underlying $\mathbb{Z}[x,x^{-1}]$ structure, and moreover we have 
\begin{equation} \label{eq:R-R^t}
\psi_B\circ R = R^t\circ \psi_A.
\end{equation}
We say a shift equivalence implemented via $(R,S)$ is \emph{unital} if either $R$ or $S$ induce an isomorphism of the dimension data that also preserves the order unit, and in this case both isomorphisms induced by $R$ and $S$ preserves the order unit. This means, in the correspondence~(\ref{thuemas}) of the isomorphism between $G_A$ and $G_B$, that we have $R^t[(1,\dots,1),0]=[(1,\dots,1),0]$, which means that for some $l\in \mathbb N$,  
\[({B^t})^lR^t(1,\dots,1)_{1\times |A|}=({B^t})^l(1,\dots,1)_{1\times |B|}.\]

\subsection{Algebras}\label{algpre}
For the basic definitions needed from ring theory, we mainly follow  \cite{Hazrat2016}. Throughout the paper rings will have identities and modules are unitary in the sense that the identity of the ring acts as the identity operator on the module. 

Let $\Gamma$ be an abelian group.
A ring $R$ is \emph{$\Gamma$-graded} if $R$ decomposes as a direct sum $\bigoplus_{\gamma\in \Gamma} R_\gamma$ where each $R_\gamma$ is an additive subgroup of $R$ and $R_\gamma R_\eta \subseteq R_{\gamma+\eta}$, for all $\gamma,\eta\in \Gamma$. 

Let $R$ be a $\Gamma$-graded ring.
A right $R$-module $M$ is \emph{$\Gamma$-graded} if $M$ decomposes as a direct sum $\bigoplus_{\gamma \in \Gamma} M_\gamma$ where each $M_\gamma$ is an additive subgroup of $M$ and $M_\gamma R_\eta \subseteq M_{\gamma+\eta}$, for all $\gamma, \eta\in \Gamma$.
A graded left $R$-module is defined analogously. 
If $R$ and $S$ are $\Gamma$-graded rings, then an $R - S$ bimodule $M$ is $\Gamma$-graded if $M = \bigoplus_{\gamma\in \Gamma} M_\gamma$ is both a graded left $R$-module and a graded right $S$-module.
In particular, $R_\gamma M_\eta S_\zeta \subseteq M_{\gamma+\eta+\zeta}$, for all $\gamma,\eta,\zeta\in \Gamma$.
If $M$ is a $\Gamma$-graded $R$-module and $\gamma\in \Gamma$, then the $\gamma$-shifted graded right $R$-module is $M(\gamma) \coloneq \bigoplus_{\eta\in \Gamma} M(\gamma)_\eta$ where $M(\gamma)_\eta = M_{\gamma+\eta}$, for all $\eta\in \Gamma$.

For a $\Gamma$-graded ring $R$, there is a graded isomorphism of rings given by the regular representation $\eta\colon R \to \End_R(R)$ 
which sends $r\in R$ to $\eta_r$ where $\eta_r(x) = rx$, for all $x\in R$, see e.g. \cite[proof of Theorem 2.3.7]{Hazrat2016}.

We write $\Modd R$ for the category of unitary right $R$-modules and module homomorphisms. If $R$ is $\Gamma$-graded, we denote by $\Gr R$ the category of unitary graded right $R$-modules with morphisms preserving the grading. For $\Gamma$-graded rings $R$ and $S$, a functor $\mathcal F \colon \Gr R \to \Gr S$ is called \emph{graded} if $\mathcal F(M(\gamma))=\mathcal F(M)(\gamma)$ for all 
$\gamma \in \Gamma$.  For $\gamma\in\Gamma$, the \emph{shift functor}
\begin{equation}\label{shiftshift}
\mathcal{T}_{\gamma}\colon \Gr R\longrightarrow \Gr R,\quad M\mapsto M(\gamma)
\end{equation}
is an isomorphism with the property that $\mathcal{T}_{\gamma}\mathcal{T}_{\eta}=\mathcal{T}_{\gamma+\eta}$, for $\gamma,\eta\in\Gamma$.

For a $\Gamma$-graded ring $R$, recall that $K_0^\gr(R)$ is the \emph{graded Grothendieck group}, which is the group completion of the  monoid consisting of the graded isomorphisms classes of finitely generated $\Gamma$-graded  projective $R$-modules equipped with the direct sum operation \cite[Section 3.1.2]{Hazrat2016}.
This group is a $\Zz[\Gamma]$-module with a structure induced from the $\Gamma$-module structure on the monoid. More specifically, if $\gamma \in \Gamma$, then for $[M]\in K_0^{\gr}(R)$ we have that $\gamma \cdot [M] \coloneq [\mathcal{T}_{\gamma}M]$ defines the action of $\Gamma$.
The class of $R$ itself defines an order unit in $K_0^\gr(R)$, and if $S$ is also a $\Gamma$-graded ring, we say that a homomorphism $K_0^\gr(R) \to K_0^\gr(S)$ is \emph{pointed (or unital)} if it takes $[R]$ to $[S]$.
Of particular interest to us is the case when $\Gamma = \Zz$ in which case $K_0^\gr(R)$ is a $\Zz[x,x^{-1}]$-module. Note that any unital graded homomorphism $\theta \colon R\to S$ of $\Zz$-graded rings $R$ and $S$, induces a pointed order-preserving $\mathbb Z[x,x^{-1}]$-module homomorphism $K_0^{\gr}(R)\to K_0^{\gr}(S)$. 

\subsection{Leavitt path algebras}

We refer the reader to \cite{Lpabook} for the definition of Leavitt path algebras and the standard terminologies. We provide some definitions for the sake of notation.

Let $E = (E^0,E^1)$ be a directed graph and $\K$ a field. The \emph{Leavitt path algebra} $L_{\K}(E)$ is the universal $\K$-algebra generated by formal elements $\{ \ v \ | \ v \in E^0 \}$ and $\{ \ e,e^* \ | \ e\in E^1 \ \}$ subject to the relations
\begin{itemize}
    \item $v^2 = v$, and $vw = 0$ if $v\neq w$ are in $E^0$
    \item $s(e) e = e r(e) = e, \ r(e)e^* = e^*s(e) = e^*$ for all $e\in E^0$
    \item $e^*e = r(e)$ for $e\in E^1$
    \item $v= \sum_{s(e) = v}ee^*$ for $v$ which is not a sink and not an infinite emitter.
\end{itemize}
Families of elements in a $\K$-algebra satisfying the above relations are called \emph{Cuntz--Krieger} $E$-families.

\subsection{Krieger's dimension triple and graded K-theory}

Let $\K$ be a field, and let $E$ be a finite graph with no sinks. In this case, $K_0^\gr(L_\K(E))$ is generated by the isomorphism classes of the form $[vL_\K(E)]$ for $v\in E^0$.
We describe the isomorphism between $K_0^\gr ( L_\K(E))$ and the dimension group $\Delta_{E}$ as proved in \cite{HazaratDyn2013} (see also \cite{Ara-Pardo}).  First, note that in \cite{HazaratDyn2013}, for a square adjacency matrix $A$, the group $\Delta_A$ was defined with $A$ acting on the left, whereas we chose the convention in \cite{Lind-Marcus2021} for which the matrix $A$ acts on the right. Suppose that $A$ is the adjacency matrix representing the graph $E$. The proof of \cite[Lemma~11]{HazaratDyn2013} gives an order $\Zz[x,x^{-1}]$-module isomorphism $\beta_E \colon K_0^\mathrm{gr} (L_\K(E)) \to G_{A}$ such that $\beta_{E} ( [ v L_\K(E) ] ) = \psi_{A}( [v, 0 ])$ for all $v \in E^0$. Thus $\beta_E$ preserves the order units, i.e., $\beta_E([L_\K(E)])=[(1,\dots,1),0]$. Consequently, we get an order $\Zz[x,x^{-1}]$-module isomorphism $\psi_{A} \circ \beta_E \colon K_0^\mathrm{gr} (L_\K(E)) \to \Delta_{A}$. 

Any shift equivalence implemented via $(R, S)$ between the adjacency matrices $A$ and $B$ of the graphs $E$ and $F$ (respectively) induce a unique order $\Zz[x,x^{-1}]$-module isomorphism $\alpha \colon K_0^\mathrm{gr} (L_\K(E)) \to K_0^\mathrm{gr} (L_\K(F))$ making the diagram
\begin{equation}\label{eq-kthy-shift}
\begin{tikzcd}
    K_0^\gr(L_\K(E)) \arrow[r, "\alpha"] \arrow[d, "\beta_E"]  & K_0^\gr(L_\K(F)) \arrow[d,"\beta_F"] \\
    G_{A} \arrow[r,"R^t"]  \arrow[d, "\psi_{A}^{-1}"] & G_{B} \arrow[d, "\psi_{B}^{-1}"]\\
    \Delta_{A} \arrow[r, "R"] & \Delta_{B}.
\end{tikzcd}
\end{equation}
commutative.  In addition, Krieger's Theorem \cite{Kr-shift} (see also \cite[Theorem 6.4]{Effros1979}), we get that for finite graphs $E$ and $F$ with no sinks, if $\alpha \colon K_0^\mathrm{gr} (L_\K(E)) \to K_0^\mathrm{gr} (L_\K(F))$ is an order $\Zz[x,x^{-1}]$-module isomorphism, then there exists a shift equivalence implemented via $(R, S)$ between $A$ and $B$ of the graphs $E$ and $F$ such that modulo the $\Zz[x,x^{-1}]$-action, $R$ induces $\alpha$.  For $n \geq 1$, as the shift equivalence $(A, A)$ induces the $\Zz[x,x^{-1}]$-action on $K_0^\mathrm{gr}(L_\K(E))$, we see that $\alpha$ is pointed if and only if there exists a shift equivalence $(R,S)$ that is unital.

\section{Pointed graded Morita equivalence} \label{sec:pointed-graded-ME}

Let $\Gamma$ be an abelian group, and let $R$ and $S$ be $\Gamma$-graded rings.
A $\Gamma$-graded $R-S$ bimodule $M$ determines a functor $\cF_M \coloneqq -\otimes_R M\colon\Gr R\longrightarrow \Gr S$ between categories of graded modules, which in particular satisfies for all $f\in \End_R(R)$ that $\cF_M(f) = f\otimes \id_M \in \End_S(R\otimes M)$. We start with the following proposition.

\begin{proposition}
\label{prop:ring-lift}
Let $M$ be a graded $R - S$ bimodule, and assume that $M$ and $S$ are graded isomorphic as graded right $S$-modules.
Then, there exists a graded unital homomorphism of rings $\xi \colon R \to S$ such that $K_0^\gr(\xi) = K_0^\gr(\cF_M)$.
Moreover, if $\cF_M$ is a graded equivalence, then $\xi\colon R\to S$ is a graded isomorphism of rings.
\end{proposition}

\begin{proof} 
Let $\psi\colon M \to S$ be a graded isomorphism of right $S$-modules.
Then, $S$ naturally inherits a left $R$-module structure and, in particular $X\otimes _R M \cong_\gr X\otimes_R S$ whenever $X$ is a finitely generated graded projective right $R$-module.
Consequently, the functor $\cF_M$ induces a pointed homomorphism $K_0^\gr(\cF_M) \colon K_0^\gr(R) \to K_0^\gr(S)$ given by $K_0^\gr(\cF_M)[X] = [X\otimes_R M]$, for all $[X]\in K_0^\gr(R)$.

Recall that multiplication defines a graded isomorphism $\mult\colon R\otimes_R M \to M$ as right $S$-modules.
We now have a graded homomorphism of rings $\xi\colon R \to S$ given as the composition
\[
\xi\colon R\overset{\eta^R}{\longrightarrow} \End_R(R) \overset{\cF_M}{\longrightarrow} \End_S(R\otimes_R M) \overset{\Ad(\mult)}{\longrightarrow} \End_S(M) \overset{\Ad(\psi)}{\longrightarrow} \End_S(S) 
\overset{(\eta^S)^{-1}}{\longrightarrow} S,
\]
where $\eta^R$ and $\eta^S$ are left regular representations of the rings on themselves.  As each map is unit-preserving, $\xi$ is a graded unital homomorphism.  This map induces a pointed homomorphism of graded $K$-theory $K_0^\gr(\xi)$, and we show now that this coincides with $K_0^\gr(\cF_M)$.
Concretely, this means that $r\in R$ is mapped to $\psi\circ \mult\circ (\eta_r\otimes \Id_M)\circ \mult^{-1}\circ \psi^{-1}$ in $\End_S(S)$; and
via $\eta^S$, we know that the latter is equal to $\eta_t^S$ for some $t\in S$, so that $\xi(r) = t$.

The homomorphism $\xi$ determines a left $R$-module structure on $S$ given by $r.s \coloneqq \xi(r)s$, for all $r\in R$ and $s\in S$,
and the homomorphism on graded K-theory induced from $\xi$ is given by the functor $-\otimes_{\xi} S$. 
We will show that, when $S$ is equipped with the left $R$-module structure as above, $\psi\colon M \to S$ is in fact an $R-S$ bimodule isomorphism,
and this will in turn imply that the functors $-\otimes_R M$ and $-\otimes_\xi S$ define the same homomorphism on graded K-theory.
Concretely, we are only left with showing that $\psi(r.m) = \xi(r)\psi(m)$ for all $r\in R$ and $m\in M$. 

Fix $r\in R$ and $m\in M$, pick $t\in S$ such that $\xi(r) = t$, and observe that since $\mult^{-1}(m) = 1\otimes m$,
\[
\xi(r)\psi(m) = \eta_t^S(\psi(m)) = \psi\circ \mult\circ (\eta_r^R\otimes \Id_M)(1\otimes m) = \psi(r. m).
\]
Finally, if $\cF_M$ is a graded equivalence, then it is invertible and the homomorphism $\xi$ is also invertible.
\end{proof}

The above inspired us to make the following natural definition of a \emph{pointed (or unital)} graded Morita equivalence of rings.
See also \cite[Example 3.1]{Ara-Pardo}.

\begin{definition}
  Let $R$ and $S$ be $\Gamma$-graded rings.
  We say $R$ and $S$ are \emph{pointed graded Morita equivalent} if there is a graded equivalence $\phi\colon\Gr R \rightarrow \Gr S$ such that $\phi(R)$ and $S$ are isomorphic as graded right $S$-modules.
\end{definition}

\begin{proposition} \label{prop:unital-graded-ME}
  Let $R$ and $S$ be $\Gamma$-graded rings.
  If $R$ and $S$ are pointed graded Morita equivalent, then $R$ and $S$ are graded isomorphic as rings.
\end{proposition}

\begin{proof}
  The left regular representation determines graded isomorphisms as graded rings $R\cong \End_R(R)$ and $S \cong \End_S(S)$.
  Choose a graded equivalence $\phi$ from Gr-$R$ to Gr-$S$ such that $\phi(R)$ and $S$ are graded isomorphic as graded $S$-modules.
  As it is explained in the proof of \cite[Theorem 2.3.7]{Hazrat2016}, $\phi$ implements a graded isomorphism of graded rings from $\End_R(R)$ to $\End_S(\phi(R))$.
  By our hypothesis, the latter is graded isomorphic as graded rings to $\End_S(S)$, and so we are done.
\end{proof}

Now consider $\Gamma$-graded rings $R$ and $S$ and a $\Gamma$-graded $R-S$ bimodule $M$ such that the graded functor $\cF_M \colon \Gr R \to \Gr S$ is a graded equivalence as in \cite[Definition~2.3.3(2)]{Hazrat2016}.  Then $\mathcal{F}_M$ implements an isomorphism on graded $K$-theory
\[
K_0(\cF_M)\colon K_0^\gr(R) \to K_0^\gr(S),
\]
where $K_0(\cF_M)[X] = [X\otimes_R M]$. By assumption we always have a graded isomorphism $M\cong R\otimes_R M$, and if $M$ implements a pointed graded Morita equivalence (as above), then $S$ and $M$ are graded isomorphic as graded $S$-modules, so $K_0(\cF_M)$ is a pointed isomorphism in the sense that $[R]$ is mapped to $[S]$. Combining this with Proposition~\ref{prop:ring-lift}, we have the following corollary. 

\begin{corollary}
    Let $R$ and $S$ be $\Gamma$-graded rings.
    Suppose there is a graded $R - S$ bimodule $M$ such that $\cF_M \colon \Gr R \to \Gr S$ is a graded equivalence that induces a pointed isomorphism $K_0^\gr(R) \to K_0^\gr(S)$.
    Then, there is a graded isomorphism of rings $\xi\colon R \to S$ and, moreover,
    the induced map $K_0^\gr(\xi)\colon K_0^\gr(R) \to K_0^\gr(S)$ coincides with $K_0^{\gr}( \cF_M)$.
\end{corollary}

\section{Lifting graded K-theory homomorphisms} \label{sec:Ktheory}

In \cite{Hazrat2013}, the third-named author conjectured that every pointed, order module homomorphism between the graded $K$-theories of two Leavitt path algebras over finite graphs is induced by a unital graded homomorphism between the Leavitt path algebras. This conjecture was independently confirmed by Arnone \cite{Arnone2023} and Va{\v s} \cite{Vas2023}. In fact, they proved that such lifts exists for a more general class of graphs: Arnone showed the lifts can be $*$-diagonal-preserving and holds for all finite graphs and Va{\v s} showed it lifts exist for all vertex-finite countable (possibly infinite) graphs.  Here, by lifting a homomorphism of graded $K$-theory to the category of graded modules, we show that an order module homomorphism between graded $K$-theories lifts to a graded homomorphism between the Leavitt path algebras.   
Our methods are very different from those of Arnone and Va{\v s} and rely on the idea of a \emph{bridging bimodule} as is used in \cite[Definition~4.12]{Abrams-Ruiz-Tomforde2023}, which were recently used by Bilich together with the second and fourth named authors in the C*-setting in \cite[Section 6]{BDR+}. These ideas go back to work of Meyer and Sehnem and their Cuntz--Pimsner homomorphism in the context of bicategories of C*-correspondences \cite{MS19}. 

\begin{lemma}\label{ARRB}
Let $E$ and $F$ be finite graphs with no sinks and with adjacency matrices $A$ and $B$, respectively.  
\begin{enumerate}[\upshape(1)]
\item Suppose there is a non-negative integral matrix $R$ such that $AR=RB$.  Then, the maps $[v,k]  \in G_A \mapsto [ R^t v , k ] \in G_B$ and $w \in \Delta_A \mapsto wR \in \Delta_B$  are order-preserving $\mathbb Z[x,x^{-1}]$-module homomorphisms from $G_A$ to $G_B$ and $\Delta_A$ to $\Delta_B$ making the diagram 
\begin{equation}\label{eq-kthy}
\begin{tikzcd}
    \Delta_{E} \arrow[r, "R"] \arrow[d, "\psi_{E}"]  & \Delta_{F} \arrow[d, "\psi_{F}"]\\
        G_E \arrow[r,"R^t"]  & G_F 
\end{tikzcd}
\end{equation}
commute.  Consequently, there exists a unique order-preserving $\mathbb Z[x,x^{-1}]$-module homomorphism $\alpha_R$ from $K^{\gr}_0(L_\K(E))$ to $K^{\gr}_0(L_\K(F))$ making \eqref{eq-kthy-shift} commute.  We call $\alpha_R$ the homomorphism induced by $R$.

\item Suppose  there exists an order-preserving $\mathbb Z[x,x^{-1}]$-module homomorphism $\alpha$ from $K^{\gr}_0(L_\K(E))$ to $K^{\gr}_0(L_\K(F))$.  Then, there is a non-negative integral matrix $R$ and there exists $r \in \Zz$ such that $AR=RB$ and
\[
\alpha (w) = {}^{x^r} \alpha_R(w)
\]
for all $w \in K^{\gr}_0(L_\K(E))$.
\end{enumerate}
\end{lemma}
\begin{proof}
We first show (1).  Define the map 
\begin{align*}
\theta\colon (\mathbb Z^{|A|}\times \mathbb N)/ \sim_{A^t} &\longrightarrow (\mathbb Z^{|B|}\times \mathbb N)/ \sim_{B^t}\\
[v,k] &\longmapsto [R^tv,k].
\end{align*}
We check the map is well-defined and a module homomorphism. Suppose $[v,k]=[w,l]$. Then there is an $m\in \mathbb N$, such that ${(A^t)}^{m-k}v={(A^t)}^{m-l}w$. Then 
$R^t{(A^t)}^{m-k}v=R^t{(A^t)}^{m-l}w$. The relation $AR=RB$ gives $R^tA^t=B^tR^t$, and thus we get  ${(B^t)}^{m-k}R^tv={(B^t)}^{m-l}R^tw$. Now by the definition, $[R^tv,k]=[R^tw,l]$. Thus $\theta$ is well-defined. Next we have 
\begin{align*}
\theta([v,k]+[w,l])&=\theta([{(A^t)}^lv+ {(A^t)}^kw, k+l])=[R^t\big ({(A^t)}^lv+ {(A^t)}^kw\big), k+l]\\
&=[R^t{(A^t)}^lv+ R^t{(A^t)}^kw, k+l]= [{(B^t)}^lR^tv+ {(B^t)}^kR^tw, k+l]\\
&=[R^tv,k]+[R^tw,l]=\theta([v,k])+\theta([w,l]).
\end{align*}
For the module structure, it is enough to show that $\theta$ preserves the action of $x$. We have 
\begin{equation*}
\theta({}^x[v,k]))=\theta([A^tv,k]))=[R^tA^tv,k]=[B^tR^tv,k]={}^x[R^tv,k]={}^x\theta([v,k]).
\end{equation*}
Since $R$ is a non-negative integral matrix, $\theta$ is an order homomorphism. 

Note that $R$ gives a homomorphism from $\Qz^{|A|}$ and $\Qz^{|B|}$.  Since 
$$
wA^k R = w R B^k
$$
and  $wA^l \in \Zz^{|A|}$ implies $wRB^l \in \Zz^{|B|}$, we have that $R$ restricts to a homomorphism from $\Delta_A$ to $\Delta_B$.  As $R$ is a non-negative integral matrix, $R \colon \Delta_A \rightarrow \Delta_B$ is order-preserving.  To prove that $R \colon \Delta_A \rightarrow \Delta_B$ is a module map, it is enough to show that $R$ preserves the action of $x$.  We have
$$
({}^x w)R = wAR = wRB = {}^x (wR).  
$$
For the commutativity of \eqref{eq-kthy}, recall that $\psi_A(v) = [ (A^t)^l v^t , l ]$, where $l \in \Nz$ such that $vA^l \in \Zz^{|A|}$ and $\psi_B(w) = [ (B^t)^k w^t , k ]$, where $k \in \Nz$ such that $vB^k \in \Zz^{|B|}$.  As $vA^s R = v R B^s$ for all $s$, a computation shows that \eqref{eq-kthy} commutes.

Next, we will show (2).  Suppose there exists an order-preserving module homomorphism $\alpha$ from $K^{\gr}_0(L_\K(E))$ to $K^{\gr}_0(L_\K(F))$.  Then we get  an order-preserving module homomorphism $\theta\colon G_A \rightarrow G_B$ such that the diagram 
\begin{equation*}
\begin{tikzcd}
    K_0^\gr(L_\K(E)) \arrow[r, "\alpha"] \arrow[d, "\beta_E"]  & K_0^\gr(L_\K(F)) \arrow[d,"\beta_F"] \\
    G_A \arrow[r,"\theta"]  & G_B
\end{tikzcd}
\end{equation*}
commutes.  Thus, it is enough to find a non-negative integral matrix $R$ and an $r \in \Zz$ such that $AR=RB$ and 
\[
\theta ([v,k]) = {}^{x^r} [ R^t v, k ]
\]
for all $[v,k ] \in G_A$.

Set \[[v_i,l_i]\coloneq \theta([e_i,0]),\]  where $e_i\in \mathbb Z^{|A|}$, $1\leq i \leq |A|$, are columns with $1$ on $i$-th entry and zero elsewhere, $v_i \in \mathbb Z^{|B|}$ and $l_i\in \mathbb N$. Since $\theta$ is order-preserving we can choose $v_i$ such that $v_i\in \mathbb N^{|B|}$.   Let $s=\sum_{j=1}^{|A|}l_j$, $s_i=\sum_{j=1, j\not = i}^{|A|} l_j$, $1\leq i \leq |A|$. Then  $[v_i,l_i]=[{(B^t)}^{s_i} v_i,s]$, for every $1\leq i \leq |A|$. Set ${R'}^t=({(B^t)}^{s_i}v_i)_{1\leq i \leq |A|}$. Clearly $R'$ is a non-negative integral $|A| \times |B|$-matrix. Since $\theta$ is a homomorphism, for $v\in \mathbb Z^{|A|}$, writing $v=(r_1, \dots, r_{|A|})^t$,  we have 
\[\theta([v,k])=\theta\big ([\sum_{j=1}^{|A|}r_je_j,k]\big)=\sum_{j=1}^{|A|} \theta([r_je_j,k])=\sum_{j=1}^{|A|} r_j\theta([e_j,k])=[{R'}^tv,k+s].\]
Next, since $\theta$ is a $\mathbb Z[x,x^{-1}]$-module homomorphism, for any $1\leq i \leq |A|$, we have, 
\[[{R'}^tA^te_i,s]= \theta([A^te_i,0])= \theta({}^x[e_i,0])={}^x\theta([e_i,0])=[B^t{R'}^te_i,s].\]
Thus we have ${(B^t)}^{\ell_i}{R'}^tA^te_i={(B^t)}^{\ell_i}{B'}^t R^te_i$, for some $\ell_i \in \mathbb N$. Choose a large enough $\ell$, so that ${(B^t)}^{\ell}{R'}^tA^te_i={(B^t)}^{\ell}B^t {R'}^te_i$ for all $1\leq i \leq |A|$. Set $R^t={(B^t)}^\ell{R'}^t$. Then $R^tA^te_i=B^tR^te_i$, for all $1\leq i \leq |A|$, implies that 
$AR=RB$ and $\theta([v,k])=[R^tv,k+s+\ell] = {}^{x^{s+\ell}} [ R^t v, k ]$, for $[v,k]\in G_A$. The proof is complete. 
\end{proof}

Let $\theta\colon K^{\gr}_0(L(E)) \rightarrow K^{\gr}_0(L(E))$ be an order-preserving $\mathbb Z[x,x^{-1}]$-module homomorphism, which, by Lemma \ref{ARRB} above is given by a matrix $R$. 
Then, $\theta$ is injective, if there is a non-negative integral matrix $S$, such that $A^m=RS$, for some $m\in \mathbb N$. 
Similarly, $\theta$ is surjective if there is a non-negative integral matrix $T$, such that $B^k=TR$, for some $k\in \mathbb N$. 
This shows that the definition of shift equivalent can be relaxed to the following: 
Two square matrices $A$ and $B$ over $\Nz$ are shift equivalent if there exist $m,k\in \Nz$ and rectangular integral matrices $R$,  $S$ and $T$ satisfying the relations
\begin{equation} \label{eq:SE2}
    A^m = R S, \quad A R = R B, \quad B^k = T R.
\end{equation}

It remains open and quite illusive whether $L_2$ and its Cuntz splice $L_{2_-}$ are isomorphic~\cite{Lpabook,willie}. Arnone and Corti\~nas~\cite{ArnoneCor23} showed that there are no unital graded homomorphisms (in particular isomorphisms) between these algebras. Using our lifting Lemma~\ref{ARRB}, the proof of this fact is immediate.

\begin{corollary}\label{lhgtr}
Let $L_n$ be the Leavitt path algebra associated to one vertex and $n$-loops and $L_{n_-}$ be the Leavitt path over the Cuntz splice of this graph. Then there are no unital graded ring homomorphisms $L_n \rightarrow L_{n_-}$, nor in the opposite direction.     
\end{corollary}
\begin{proof}
A unital ring homomorphism $L_n \rightarrow L_{n_-}$ induces a homomorphism on the level of graded $K_0$-groups. Then by Lemma~\ref{ARRB}(2), there is a non-negative integral $1\times 3$ matrix $R$ such that $(n) R=RB$, which induces the homomorphism of $K_0$-groups. Here 
$B=\begin{pmatrix}
n & 1 & 0\\
1 & 1 & 1\\
0& 1& 1\\
\end{pmatrix}$ is the adjacency matrix of $L_{n_-}$. It immediately follows that $R$ has to be a zero vector, a contradiction. The opposite direction is similar. 
\end{proof}

\begin{remark}
In~\cite{ArnoneCor23} it was also shown that there are no unital $\mathbb Z/m \mathbb Z$-graded ring homomorphism $L_n \rightarrow L_{n_-}$, nor in the opposite direction. We briefly mention that this can also be proved similarly by re-writing Lemma~\ref{ARRB} for the Leavitt path algebras with $\mathbb Z/m \mathbb Z$-grading. Assigning weight $1\in \mathbb Z/m \mathbb Z$ to edges, make a Leavitt path algebra a $\mathbb Z/m \mathbb Z$-graded algebra. One can observe that $K_0^{\mathbb Z/m \mathbb Z-\gr}(L_\K(E))$ is the group $\mathbb Z^{|E|}\times \mathbb Z/m \mathbb Z/\sim$, where $(v,k)\sim (w,l)$ if $(A^t)^pv=(A^t)^qw$, for some $p,q$ such that  $p+k = q+l \mod m$. Here the graded Grothendieck group $K_0^{\mathbb Z/m \mathbb Z-\gr}$ is constructed from the finitely generate $\mathbb Z/m\mathbb Z$-graded  projective $L_\K(E)$-modules (see \S\ref{algpre}).  

A similar proof as in Lemma~\ref{ARRB} shows that a module homomorphism \[K_0^{\mathbb Z/m \mathbb Z-\gr}(L_\K(E))\rightarrow K_0^{\mathbb Z/m \mathbb Z-\gr}(L_\K(F)),\] gives a non-negative integral matrix $R$ such that $ARB^{km}=BR$, for some $k\in \mathbb N$. With this, Corollary~\ref{lhgtr} can be written for the $\mathbb Z/m\mathbb Z$-grading.  
\end{remark}

Let $E$ be a finite graph with no sinks, and let $\K$ be a field.
We let $\K E^0$ be the $\K$-algebra with basis $E^0$ and relation $v w = \delta_{v,w} v$, for all $v,w\in E^0$.
The $\K$-vector space $\K E^1$ with basis $E^1$ is a $\K E^0$-bimodule with the relation $v\cdot e \cdot w = \delta_{v,s(e)} \delta_{w,r(e)} e$, for all $e\in E^1$ and $v,w\in E^0$.

\medskip

\begin{proposition} \label{prop:se-mod-hom}
Let $E$ and $F$ be finite graphs with no sinks and with adjacency matrices $A$ and $B$, respectively. 
If there is a non-negative integral matrix $R$ such that $AR=RB$, then there is an $L_\K(E) - L_\K(F)$-bimodule $M$, called the bridging bimodule, and the functor 
\begin{equation}\label{strip}
\cF_M \coloneq - \otimes_{L_\K(E)} M  \colon \Gr L_\K(E) \to 
\Gr L_\K(F),    
\end{equation}
induces an order-preserving  $\Zz[x,x^{-1}]$-module homomorphism from $K_0^\gr(L_\K(E))$ to $K_0^\gr(L_\K(F))$ that, up to identification of the graded K-theory with the dimension triple, coincides with the map induced from $R$. 
\end{proposition}

\begin{proof}
Consider the $\K$-vector space $\K R^1$ with basis 
$$
\big \{ \mathbf{e}_{v, w, i} : v \in V, w \in W \text{ with } R(v,w) \neq 0,  1 \leq i \leq R(v,w) \big \} 
$$
which is a $\K E^0$--$\K F^0$-bimodule via $v' \mathbf{e}_{v, w, i} w' = \delta_{v, v'} \delta_{w, w'} \mathbf{e}_{v,w,i}$.  The relation $AR = R B$ implies there is an isomorphism $\sigma\colon \K E^1 \otimes_{\K E^0} \K R^1 \cong \K R^1 \otimes_{\K F^0} \K F^1$ (this is a specified conjugacy in the terminology of \cite[Section 4]{Abrams-Ruiz-Tomforde2023}).
By \cite[Theorem 4.11]{Abrams-Ruiz-Tomforde2023}, it now follows that the right $L_\K(F)$-module \[M \coloneq \K R^1 \otimes_{\K F^0} L_\K(F),\] can be endowed with a left $L_\K(E)$-module structure for which $M$ becomes a $L_\K(E)$--$L_\K(F)$-bimodule.  
Moreover, if we set the grading $M_n = \K R^1 \otimes_{\K F^0} L_\K(F)_n$, for all $n \in \Zz$, then $M$ is a graded $L_\K(E)$--$L_\K(F)$-bimodule.
This is called the bridging bimodule.

Next, we verify that the functor $\cF_M \coloneqq - \otimes_{L_\K(E)} M$ induces a homomorphism from $K_0^\gr(L_\K(E))$ to $K_0^\gr(L_\K(F))$.  As $\cF_M$ is a functor and as $K_0^\mathrm{gr} (L_\K(E))$ is generated by $[ v L_\K(E)(n)]$ for all $v \in E^0$, it is enough to show $\cF( (vL_\K(E))(n)  )$ is a finitely generated graded projective module for all $v \in E^0$. We claim that for each $v\in E^0$ and $n\in \Zz$ there is a graded isomorphism of right $L_\K(F)$-modules
\[
\alpha\colon (vL_\K(E))(n) \otimes_{L_\K(E)} M \longrightarrow\bigoplus_{w\in F^0} (wL_\K(F))(n)^{R(v,w)}.
\]
To see this first note
$$
v(L_\K(E))(n) \otimes_{L_\K(E)} M \cong vM(n)
$$
via the graded module isomorphism $vS \otimes m \otimes T \mapsto (vS) \cdot (m \otimes T)$ with inverse $v ( m \otimes T) \mapsto v \otimes m \otimes T$.  
Since
$$
v(\K R^1) \cong \bigoplus_{ \substack{ w \in F^0  \\ R(v,w) \neq 0} } \bigoplus_{ i = 1}^{ R(v,w) } \K \mathbf{e}_{v,w,i}
$$ 
as right $\K F^0$-modules, $\alpha$ will be the composition of the following graded module isomorphisms
\begin{align*}
    v(L_\K(E))(n) \otimes_{L_\K(E)} M &\cong vM(n) = v( \K R^1) \otimes_{\K F^0 } L_\K(F)(n) \\
                                &\cong \bigoplus_{ \substack{ w \in F^0  \\ R(v,w) \neq 0} } \bigoplus_{ i = 1}^{ R(v,w) } \K \mathbf{e}_{v,w,i} \otimes_{ \K F^0 } L_\K(F)(n) \\
                                &\cong \bigoplus_{ \substack{ w \in F^0  \\ R(v,w) \neq 0} } (wL_\K(F))(n)^{ R(v,w)}
\end{align*}
This shows that $\cF_M$ induces a $\Zz[x,x^{-1}]$-module homomorphism $K_0^\gr(L_\K(E)) \to K_0^\gr(L_\K(F))$.  Moreover, $\cF_M$ clearly sends $K_0^\gr(L_\K(E))_+$ to $K_0^\gr(L_\K(F))_+$.

For the second assertion, recall the notation from Section \ref{sec:symbolic-dynamics}. Since diagram \eqref{eq-kthy-shift} with $\alpha = K_0^\gr(\cF_M)$ commutes, we see that $\cF_M$ induces the same map as $R$ as a map from $\Delta_A$ to $\Delta_B$.
\end{proof}

Notice that due to the commutation of diagram \eqref{eq-kthy-shift} with $\alpha = K_0^\gr(\cF_M)$ we also have that $K_0^\gr(\cF_M)$ is pointed precisely when the integral matrix $R$ induces a unital map between $\Delta_A$ to $\Delta_B$, which happens exactly when the shift equivalence is unital. 
Thus, in the case of a unital shift equivalence, we can lift the homomorphism on graded $K$-theory to a graded homomorphism of algebras using Proposition \ref{prop:ring-lift}.

\begin{theorem}\label{homfunlift}
    Let $E$ and $F$ be finite graphs with no sinks. 
    \begin{enumerate}[\upshape(1)]

    \item If $\mathcal F\colon \Gr L_\K(E) \to 
\Gr L_\K(F)$ is a graded functor that is right exact and commutes with direct sums, then $\mathcal F \cong  - \otimes_{L_\K(E)} \mathcal F(L_\K(F))$. Thus, if 
$\mathcal{F}(L_\K(F))$ is a graded, finitely generated, projective $L_\K(F)$-module, then $\mathcal{F}$ induces an order-preserving $\mathbb Z[x,x^{-1}]$-module homomorphism $K^{\gr}_0(L_\K(E)) \rightarrow K^{\gr}_0(L_\K(F))$ which is pointed if $\mathcal F$ is. 

\item If $\alpha$ is an order-preserving $\mathbb Z[x,x^{-1}]$-module homomorphism $K^{\gr}_0(L_\K(E)) \rightarrow K^{\gr}_0(L_\K(F))$, then there exists a graded functor $\mathcal F\colon \Gr L_\K(E) \to \Gr L_\K(F)$ which is right exact, commutes with direct sums, and $\mathcal{F}(L_\K(E))$ is a graded, finitely generated, projective $L_\K(F)$-module such that $K_0^\gr(\mathcal{F}) = \alpha$. 
  Furthermore, if $\alpha$ is pointed, so is $\mathcal F$.  \end{enumerate}
\end{theorem}

\begin{proof}
(1) Let $\mathcal F\colon \Gr R \to \Gr S$ be a graded functor between the module categories of $\Gamma$-graded rings $R$ and $S$, where $\Gamma$ is an arbitrary abelian group.  
We will show that the functor $\mathcal F$ is isomorphic to $- \otimes_R \mathcal F(R)$. 
Specializing to $\mathbb Z$-graded Leavitt path algebras, we will then obtain one direction of the theorem. The proof is similar to the non-graded case~\cite{watts} with extra attention given to the presence of the grading. We provide the proof for completeness. 

Let $M$ be a graded right $R$-module and $m\in M_\alpha$. The map $\theta_m\colon R \to M(\alpha); a \mapsto ma$, is a graded right $R$-module homomorphism, thus a morphism of the category $\Gr R$. Hence we have a graded right $S$-module homomorphism $\mathcal F(\theta_m)\colon\mathcal F(R) \to \mathcal F(M)(\alpha)$. In particular, for any $a\in R_\alpha$ we have $\mathcal F(\theta_a)\colon \mathcal F(R) \to \mathcal F(R)(\alpha)$. For $a\in R_\alpha$ and $b\in \mathcal F(R)_\beta$, define 
$a.b=\mathcal F(\theta_a)(b)$ and extend it linearly to all $R$ and $\mathcal F(R)$. It is easy to see that this makes $\mathcal F(R)$ a graded left $R$-module compatible with the graded right $S$-module structure. Define 
\begin{align*}
M\otimes_A \mathcal F(R) &\longrightarrow \mathcal F(M)\\
m\otimes x &\longmapsto \mathcal F(\theta_m)(x).
\end{align*}
This is a natural graded right $S$-module homomorphism, and $R\otimes_R \mathcal F(R) \cong \mathcal F(R)$.

We write the graded right $R$-module $M$ as a cokernel of the graded free $R$-modules,   that is 
\begin{equation}\label{gftrgu}
\bigoplus_{j\in J}R(\beta_j)\longrightarrow\bigoplus_{i\in I} R(\alpha_i)
\stackrel{\varepsilon}{\longrightarrow} M \longrightarrow 0,
\end{equation}
is an exact sequence of graded right $R$-modules. Here, for $\varepsilon$, choose a set $\{m_i\}_{i\in I}$ of homogeneous generators of $M$, with 
$\deg(m_i)=-\alpha_i$ and define $\varepsilon(e_i) \coloneq m_i$, where $e_i\in \bigoplus_{i\in I}R(\alpha_i)$ with $1$ on the $i$-th entry and zero elsewhere. One repeats this for $\ker(\varepsilon)$ to obtain the exact sequence~(\ref{gftrgu}).

Applying the functors $\mathcal F$ and the tensor product to this sequence we have:

\begin{equation*}
\begin{tikzcd}
    \bigoplus_{j\in J}R(\beta_j) \otimes_R \mathcal F(R) \arrow[r] \arrow[d]  & \bigoplus_{i\in I}R(\alpha_i) \otimes_R \mathcal F(R) \arrow[d] \arrow[r] & M\otimes_R \mathcal F(R)  \arrow[d] \arrow[r] & 0 \\
\mathcal F \big (\bigoplus_{j\in J}R(\beta_j) \big)  \arrow[r]   & \mathcal F \big(\bigoplus_{i\in I}R(\alpha_i)\big) \arrow[r] & \mathcal F(M)\arrow[r] & 0 
\end{tikzcd}
\end{equation*}

  Since the functor $\mathcal F$ and the tensor product are right exact, the rows in the above diagram are exact. Since the graded functor $\mathcal F$ commutes with direct sum, the left two vertical maps are isomorphisms, which implies that the right vertical map is an isomorphism too. Since all the isomorphisms are natural, this shows that $\mathcal F \cong -\otimes_R \mathcal F(R)$. Thus $\mathcal F$ induces a homomorphism $K_0^{\gr}(R) \rightarrow K_0^{\gr}(S)$, where we have $ [P]\mapsto [\mathcal{F}(P)] = [P\otimes_R \mathcal F(R)]$, when $\mathcal{F}(R)$ is a graded, finitely generated, projected module.   
  Clearly if $\mathcal F$ is pointed, meaning $\mathcal F(R) \cong S$ as graded $S$-modules,  then under the above homomorphism  $ [R]\mapsto [R\otimes_R \mathcal F(R)]=[S]$, i.e., the homomorphism is pointed. 
  
  (2) Let $A$ and $B$ be the adjacency matrices of the graphs $E$ and $F$, respectively. Given an order-preserving $\mathbb Z[x,x^{-1}]$-module $\alpha: K_0^{\gr}(L_\K(E)) \rightarrow K_0^{\gr}(L_\K(F))$, by Lemma~\ref{ARRB}, there is a non-negative integral matrix $R$ and $r\in \mathbb Z$, such that $AR=RB$, and $\alpha$ is induced by $R$ and the shift $r$.  Now by Proposition~\ref{prop:se-mod-hom} there is a graded $L_\K(F)$-module $M$ such that $\mathcal F_M = -\otimes_{L_\K(E)} M$ induces a homomorphism on $K_0^{\gr}$ groups that coincides with $R$. Thus $\mathcal F_{M(r)}= -\otimes_{L_\K(E)} M(r)$ induces the homomorphism $\alpha$. The pointed claim is obvious now.
\end{proof}

We can now provide a simple proof of the fullness conjecture made by the third-named author in \cite{Hazrat2013} for finite graphs with no sinks. 

\begin{corollary}
Let $E$ and $F$ be finite graphs with no sinks.  Suppose there exists an order-preserving, pointed, $\Zz[x,x^{-1}]$-module homomorphism $\alpha \colon K_0^{\mathrm{gr}} (L_\K(E)) \to K_0^{\mathrm{gr}} (L_\K(F))$.  Then, there exists a unital graded homomorphism $\psi \colon L_\K(E) \to L_\K(F)$ such that $K_0(\psi) = \alpha$.
\end{corollary}
\begin{proof}
Theorem~\ref{homfunlift}  guarantees a bimodule $M$ and the functor $-\otimes_{L_\K(E)} M$ between the module categories of the Leavitt path algebras. Since the homomorphism is pointed,  $M$ and $L_\K(F)$ are isomorphic as $L_\K(F)$-modules. Thus Proposition~\ref{prop:ring-lift} gives a graded homomorphism   $\psi \colon L_\K(E) \to L_\K(F)$ as claimed. 
\end{proof}

\section{Unital aligned module shift equivalence} \label{sec:aligned}

Let $E$ and $F$ be finite graphs with no sinks. The aim of this section is to show that if there is a pointed $\mathbb Z[x,x^{-1}]$-module isomorphism $\alpha \colon K_0^{\mathrm{gr}} (L_\K(E)) \to K_0^{\mathrm{gr}} (L_\K(F))$ which is induced by an \emph{aligned} unital module shift equivalence then $\alpha$ lifts to a graded isomorphism between the Leavitt path algebras, thus confirming the Graded Classification Conjecture ~\ref{gcc2013}, under these assumptions. 

We first recall some definitions from \cite{Abrams-Ruiz-Tomforde2023}. In what follows, given two finite vertex sets $E^0$ and $F^0$, we say that a finite set $G^1$ is an \emph{edge set} between $E^0$ and $F^0$ if it comes equipped with range and source maps $r\colon G^1 \rightarrow F^0$ and $s\colon G^1 \rightarrow E^0$. The $5$-tuple $(E^0, F^0, G^1,r,s)$ is called a \emph{polymorphism} in \cite{Abrams-Ruiz-Tomforde2023}.  The vector space $\K G^1$ with basis $G^1$ becomes a $\K E^0 - \K F^0$-bimodule by the rule $v \cdot e \cdot w = \delta_{v, s(e)} \delta_{w, r(e)} e$.  Note that rectangular $E^0 \times F^0$ integral matrices $R$ naturally correspond to edge sets $G_R^1$ from $E^0$ to $F^0$.

\begin{definition} \label{def:module-se}
Two finite graphs $E$ and $F$ with no sinks are \emph{module shift equivalent} if there are edge sets $G^1$ from $E^0$ to $F^0$ and $H^1$ from $F^0$ to $E^0$,  respectively, together with an $m\in \Nz_{>0}$, and bimodule isomorphisms
\begin{align}
    \omega_E&\colon \K G^1 \otimes_{\K F^0} \K H^1 \longrightarrow {(\K E^1)}^{\otimes m} & \omega_F&\colon  \K H^1 \otimes_{\K E^0} \K G^1 \longrightarrow {(\K F^1)}^{\otimes m}  \\
\sigma_G &\colon {\K E^1} \otimes_{\K E^0} \K G^1 \longrightarrow \K G^1 \otimes_{\K F^0} {\K F^1} & \sigma_H &\colon {\K F^1} \otimes_{\K F^0} \K H^1 \longrightarrow \K H^1 \otimes_{\K E^0} {\K E^1}.
\end{align}
\end{definition}

Abrams, the fourth-named author, and Tomforde show that the adjacency matrices of $E$ and $F$ are shift equivalent via the  non-negative integral rectangular matrices $(R,S)$ (as in the classical sense from symbolic dynamics; see (\ref{eq:SE})) if and only if $E$ and $F$ are module shift equivalent \cite[Theorem 3.9]{Abrams-Ruiz-Tomforde2023}.  Moreover, by \cite[Theorem 3.7]{Abrams-Ruiz-Tomforde2023}, every $\K E^0- \K F^0$-bimodule is isomorphic to $\K G^1$ for some unique edge set $G^1$ between $E^0$ and $F^0$.  Consequently, there is no new information when one replaces the bimodules of edge sets in Definition~\ref{def:module-se} by arbitrary bimodules.

\begin{definition}
Two finite graphs $E$ and $F$ with no sinks are  \emph{aligned module shift equivalent} if there is a module shift equivalence as in Definition \ref{def:module-se} and, in addition, the \emph{associator relations} are satisfied

\begin{equation}
    \label{}
\begin{tikzcd}
    \K E^1\otimes_{\K E^0} \K G^1 \otimes_{\K F^0} \K H^1 \arrow[r,"\sigma_G \otimes \id"] \arrow[d,"\id\otimes \omega_E"] & \K G^1 \otimes_{\K F^0} \K F^1 \otimes_{\K F^0} \K H^1 \arrow[r,"\id \otimes \sigma_H"] & \K G^1 \otimes_{\K F^0} \K H^1 \otimes_{\K E^0} \K E^1 \arrow[d,"\omega_E \otimes \id"] \\
    \K E^1\otimes_{\K E^0} (\K E^1)^{\otimes m} \arrow[rr,"\nu^m_E"] & & (\K E^1)^{\otimes m}\otimes_{\K E^0} \K E^1
\end{tikzcd}
\end{equation}
where $\nu^m_E\colon x\otimes (x_1\otimes \cdots \otimes x_m) \mapsto (x\otimes x_1\otimes \cdots \otimes x_{m-1})\otimes x_m$,
and
\begin{equation}
    \label{}
\begin{tikzcd}
    \K F^1\otimes_{\K F^0} \K H^1 \otimes_{\K E^0} \K G^1 \arrow[r,"\sigma_H \otimes \id"] \arrow[d,"\id\otimes \omega_F"] & \K H^1 \otimes_{\K E^0} \K E^1 \otimes_{\K E^0} \K G^1 \arrow[r,"\id \otimes \sigma_G"] & \K H^1 \otimes_{\K E^0} \K G^1 \otimes_{\K F^0} \K F^1 \arrow[d,"\omega_F \otimes \id"] \\
    \K F^1\otimes_{\K F^0} (\K F^1)^{\otimes m} \arrow[rr,"\nu^m_F"] & & (\K F^1)^{\otimes m}\otimes_{\K F^0} \K F^1
\end{tikzcd}
\end{equation}
where $\nu^m_F$ is defined similarly, that is, the diagrams commute.
\end{definition}

\begin{remark}
    In a forthcoming work of the second and fourth named authors with Bilich, we define a relation on the class of finite essential matrices called \emph{aligned shift equivalence}.  We prove two finite essential matrices $A$ and $B$ are aligned shift equivalent if and only if they are strong shift equivalent.  Moreover, we show that aligned shift equivalence implies a C*-correspondence analogue of module aligned shift equivalence of the associated directed graphs. The converse remains open.  
\end{remark}

\begin{definition}
    Two finite graphs $E$ and $F$ with no sinks are \emph{unitally aligned module shift equivalent} if there are rectangular integral matrices $R$ and $S$, one of which induces a unital map on dimension groups, and such that $R$ and $S$ correspond to edge sets $G_R^1$ between $E^0$ and $F^0$ and $G_S^1$ between $F^0$ and $E^0$ (respectively), so that $E$ and $F$ are aligned module shift equivalent via $(G_R^1,G_S^1)$. 
\end{definition}

\begin{remark}
We know that strong shift equivalence implies aligned module shift equivalence which implies shift equivalence~\cite[Proposition~6.9]{Abrams-Ruiz-Tomforde2023}. 
At the moment however, we do not know if aligned module shift equivalence implies strong shift equivalence.    
\end{remark}

Before proving the main result of the paper, we need a lemma which confirms \cite[Conjecture~2]{Hazrat2013}.
It says that a ring homomorphism between Leavitt path algebras may be exchanged for an algebra homomorphism in a way that does not affect the induced map on graded $K$-theory. The lemma below is valid for any grading on Leavitt path algebra $L_\K(E)$ as long as the homogeneous components are $\K$-modules.
The proof is stated for finite graphs since we need to keep track of the unit for our application below, but for the argument for the part showing that ring homomorphisms extends to algebra homomorphism we need no  assumptions on the graphs.

\begin{lemma} \label{lem:ring-to-algebra-isomorphism}
    Let $E$ and $F$ be finite graphs, and let $\K$ be a field.  Suppose $\psi \colon L_\K(E) \to L_\K(F)$ is a unital graded homomorphism (for some grading by an abelian group) of rings.  Then there exists a unital graded homomorphism of $\K$-algebras $\Psi \colon L_\K(E) \to L_\K(F)$ such that the induced homomorphisms on graded $K$-theory $K_0^{\mathrm{gr}}(\psi)$ and $K_0^{\mathrm{gr}}(\Psi)$ are equal.  Moreover, if $\psi$ is a graded ring isomorphism, then $\Psi$ is a graded algebra isomorphism.
\end{lemma}

\begin{proof}
    Observe that $\{\psi(v), \psi(e),\psi(e^*) : v\in E^0, e\in E^1\}$ is a Cuntz--Krieger $E$-family in $L_\K(F)$.  By the universal property of Leavitt path algebras, there is a $\K$-algebra homomorphism $\Psi \colon L_\K(E) \to L_\K(F)$ such that $\Psi(v) = \psi(v)$, $\Psi(e) = \psi(e)$, $\Psi(e^*) = \psi(e^*)$, for all $v \in E^0$ and  $e \in E^1$.  As $\psi$ is a graded homomorphism, $\Psi$ is a graded homomorphism.  
    
    To show that $K_0^{\mathrm{gr}}(\psi)$ and $K_0^{\mathrm{gr}}(\Psi)$ coincide, it is enough to show that they agree on the generators $[v L_\K(E)]$ for all $v \in E^0$ of $K_0^\mathrm{gr}(L_\K(E))$.  Let ${}_{\psi} L_\K(F)$ be the graded $L_\K(E)$--$L_\K(F)$-bimodule with left action given by $a \cdot x = \psi(a)x$.  Similarly, ${}_{\Psi} L_\K(F)$ will denote the graded $L_\K(E)$--$L_\K(F)$-bimodule with left action given by $a \cdot x = \Psi(a)x$.  Then 
    $$
    K_0^{\mathrm{gr}}(\psi) ( [ v L_\K(E) ] ) = [  v L_\K(E) \otimes_{ L_\K(E) } {}_{\psi} L_\K(F) ] 
    $$
    and 
    $$
    K_0^{\mathrm{gr}}(\Psi) ( [ v L_\K(E) ] ) = [  v L_\K(E) \otimes_{ L_\K(E) } {}_{\Psi} L_\K(F) ].
    $$
    Hence, it is enough to show that $v L_\K(E) \otimes_{ L_\K(E) } {}_{\psi} L_\K(F)$ and $v L_\K(E) \otimes_{ L_\K(E) } {}_{\Psi} L_\K(F)$ are graded isomorphic.  This follows as the restriction of the graded isomorphisms 
    \begin{align*}
        L_\K(E) \otimes_{ L_\K(E) } {}_{\psi} L_\K(F) \longrightarrow L_\K(F) \quad \text{and} \quad L_\K(E) \otimes_{ L_\K(E) } {}_{\Psi} L_\K(F) \longrightarrow L_\K(F)
    \end{align*}    
    to $v L_\K(E) \otimes_{ L_\K(E) } {}_{\psi} L_\K(F)$ and $v L_\K(E) \otimes_{ L_\K(E) } {}_{\Psi} L_\K(F)$ have image $\psi(v)L_\K(F)$ and $\Psi(v)L_\K(F)$, respectively.  As $\psi(v)L_\K(F)=\Psi(v)L_\K(F)$, we get $K_0^{\mathrm{gr}}(\psi)$ and $K_0^{\mathrm{gr}}(\Psi)$ agree on the generators of $K_0^\mathrm{gr}$, and hence they are equal.

    Now assume that $\psi$ is an isomorphism.  Then $\{\psi^{-1}(w), \psi^{-1}(f),\psi^{-1}(f^*) : w\in F^0, f\in F^1\}$ is a Cuntz--Krieger $F$-family in $L_\K(E)$. 
    By the universal property of Leavitt path algebras, there exists a $\K$-algebra homomorphism $\beta\colon L_\K(F) \to L_\K(E)$ satisfying $\beta(w) = \psi^{-1}(w)$, $\beta(f) = \psi^{-1}(f)$, and $\beta(f^*) = \psi^{-1}(f^*)$, for all $w\in F^0$ and $f\in F^1$.
    In fact, since $\psi^{-1}$ is a graded homomorphism, it follows that $\beta$ is a graded homomorphism.
    It is now straightforward to see that $\Psi$ and $\beta$ are inverses of each other and therefore $\Psi$ is a graded isomorphism of algebras.
\end{proof}

We are in a position to extend a theorem of Ara and Pardo  \cite[Theorem 3.15]{Ara-Pardo}.

\begin{corollary} \label{thm:graded-isomorphism}
  Let $E$ and $F$ be finite graphs with no sinks, and let 
  $\K$ be a field. 
  If $E$ and $F$ are unitally aligned  module shift equivalent via matrices $(R, S)$, then the Leavitt path algebras $L_\K(E)$ and $L_\K(F)$ are graded isomorphic as $\K$-algebras via an isomorphism that induces the maps $(R, S)$ between dimension triples.
\end{corollary}

\begin{proof}
If $E$ and $F$ are unitally aligned shift equivalent via $(R,S)$, then for $M \coloneq \K G_R^1$ the functor $\cF_M$ from Proposition \ref{prop:se-mod-hom} induces a  pointed isomorphism  $K_0^\gr(L_\K(E)) \to K_0^\gr(L_\K(F))$ and by \cite[Theorem 6.7]{Abrams-Ruiz-Tomforde2023}, the bridging bimodule $M$ induces a graded equivalence of categories from Mod-$L_\K(E)$ to Mod-$L_\K(F)$, so $\cF_M$ is an isomorphism.
As the shift equivalence is also unital, this means that $\cF_M$ is pointed.
We may now invoke Proposition \ref{prop:ring-lift} to find a graded isomorphism of rings $\xi\colon L_\K(E) \to L_\K(F)$ that induces the same map as $\cF_M$ on graded $K$-theory.
By Lemma \ref{lem:ring-to-algebra-isomorphism}, $\xi$ can be chosen to be a graded isomorphism of algebras.
\end{proof}

\end{document}